\documentclass[3p,times]{elsarticle}%
\usepackage[colorlinks=true,urlcolor=blue,pdfstartview=FitH, linkcolor=blue,citecolor=blue,bookmarksopen,bookmarksnumbered={true}]{hyperref}
\usepackage{enumerate}
\usepackage{amsmath}
\usepackage{amsfonts}
\usepackage{amssymb}
\usepackage{amsthm}
\biboptions{comma,square,numbers,sort&compress}
\journal{...}
\newtheorem{theorem}{Theorem}[section]
\newtheorem{corollary}{Corollary}[section]
\newtheorem{lemma}{Lemma}[section]

\newdefinition{remark}{Remark}[section]
\newdefinition{definition}{Definition}[section]
\newdefinition{example}{Example}[section]

\begin{document}
%%%
\begin{frontmatter}
\title{Fixed point theorems for $\alpha$--contractive mappings of Meir--Keeler type and applications}
\author[mb]{Maher Berzig\corref{cor1}}
\cortext[cor1]{Corresponding author}
\ead{maher.berzig@gmail.com}
\author[mdr]{Mircea-Dan Rus}
\ead{rus.mircea@math.utcluj.ro}
\address[mb]{Department of Mathematics, Tunis College of Sciences and Techniques, 5 Avenue Taha Hussein, Tunis, Tunisia}
\address[mdr]{Department of Mathematics, Technical University of Cluj-Napoca, Str. Memorandumului 28, 400114 Cluj-Napoca, Romania}
\begin{abstract}
In this paper, we introduce the notion of $\alpha$--contractive mapping of Meir--Keeler type in complete metric spaces and prove new theorems which assure the existence, uniqueness and iterative approximation of the fixed point for this type of contraction. The presented theorems extend, generalize and improve several existing results in literature.  To validate our results, we establish the existence and uniqueness of solution to a class of third order two point boundary value problems.
\end{abstract}
\end{frontmatter}
%%%%%%%%%%%%%%%%%%%%%%%%%%%%%%%%%%%%%%%%%%%%%%%%%%%

\section{Introduction}

In \cite{Meir1969}, Meir and Keeler introduced a new contraction condition for
self-maps in metric spaces and generalized the well known Banach contraction
principle as follows.

\begin{theorem}
[\cite{Meir1969}]\label{Th:0}Let $(X,d)$ be a complete metric space and
$T:X\rightarrow X$. Assume that for every $\varepsilon>0$, there exists
$\delta(\varepsilon)>0$ such that:
\[
x,y\in X:\varepsilon\leq d(x,y)<\varepsilon+\delta(\varepsilon)\Rightarrow
d(Tx,Ty)<\varepsilon.
\]
Then $T$ has a unique fixed point $x^{\ast}\in X$ and $T^{n}x\rightarrow
x^{\ast}$ (as $n\rightarrow\infty$) for every $x\in X$, where $T^{n}$ denotes
the $n$-th order iterate of $T$.
\end{theorem}

In another direction, Ran and Reurings \cite{Ran2004} extended Banach's
contraction principle to the setting of ordered metric spaces and obtained
some interesting applications to matrix equations. Later on, the results of
Ran and Reurings were extended and generalized by many authors (e.g.,
\cite{Agarwal2008,Berzig2012a,Berzig2012b,Bhaskar2006,Kirk2003,Nieto2005,Nieto2007,SametT2012,Samet2010,Rus2011}
and the references therein). In particular, Harjani et al. \cite{Harjani2011}
unified these two directions by studying the fixed points of Meir--Keeler type
contractions in ordered metric spaces.

Very recently, Samet et al.\ \cite{SametVV2012} took a new approach to the
generalization of Banach's contraction principle and introduced the concept of
$\alpha-\psi$--contractive type mappings, while establishing various fixed
point theorems for such mappings in the setting of complete metric spaces. In
particular, this new approach contains many of the generalizations considered
in
\cite{Agarwal2008,Berzig2012a,Berzig2012b,Bhaskar2006,Kirk2003,Nieto2005,Nieto2007,SametT2012,Samet2010,Rus2011,Harjani2011,Ran2004}
as special cases.

In this context, the aim of this paper is to unify the concepts of
Meir--Keeler contraction \cite{Meir1969} and $\alpha-\psi$--contractive type
mapping \cite{SametVV2012} and establish some new fixed point theorems in
complete metric spaces for such mappings. Several consequences of our results
are presented in Section \ref{Sec:3}. We validate our results with an
application to the study of the existence and uniqueness of solutions for a
class of third order two point boundary value problems.

%%%%%%%%%%%%%%%%%%%%%%%%%%%%%%%%%%%%%%%%%%%%%%%%%%%%%%%%%%%%%%%%%%%%%%%%%%%%%

\section{Main results\label{Sec:2}}

%%% ------------------------------------------------------------------------

\subsection{Preliminaries}

Throughout this paper, let $\mathbb{N}$ denote the set of all non-negative
integers, $\mathbb{Z}$ the set of all integers and $\mathbb{R}$ the set of all
real numbers. We start by introducing the concept of $\alpha$--contractive
mapping of Meir--Keeler type. Subsequently, we prove some lemmas useful later.

In what follows, let $(X,d)$ be a metric space, $T:X\rightarrow X$ and
$\alpha:X\times X\rightarrow\lbrack0,+\infty)$, if not stated otherwise.

\begin{definition}
We say that $T$ is an $\alpha$--contractive mapping of Meir--Keeler type (with
respect to $d$) if for all $\varepsilon>0$, there exists $\delta
(\varepsilon)>0$ such that
\begin{equation}
x,y\in X:\varepsilon\leq d(x,y)<\varepsilon+\delta(\varepsilon)\Rightarrow
\alpha(x,y)d(Tx,Ty)<\varepsilon. \label{eq:01}%
\end{equation}

\end{definition}

\begin{lemma}
\label{Lemma:1}If $T$ is an $\alpha$--contractive mapping of Meir--Keeler
type, then
\[
\alpha(x,y)d(Tx,Ty)<d(x,y)\quad\text{for all }x,y\in X\text{ with }x\neq y.
\]

\end{lemma}

\begin{proof}
Fix $x,y\in X$ with $x\neq y$ and let $\varepsilon:=d(x,y)>0$. Then, by
(\ref{eq:01}), $\alpha(x,y)d(Tx,Ty)<\varepsilon=d(x,y)$, which concludes the proof.
\end{proof}

\begin{definition}
[\cite{SametVV2012}]We say that $T$ is $\alpha$--admissible if
\[
x,y\in X:\alpha(x,y)\geq1\Rightarrow\alpha(Tx,Ty)\geq1.
\]

\end{definition}

\begin{example}
Let $X=\mathbb{R}$. Define $\alpha:X\times X\rightarrow\lbrack0,+\infty)$ by
\begin{equation}
\alpha(x,y)=\left\{
\begin{array}
[c]{ll}%
\mathrm{e}^{x-y} & \text{if}\quad x\geq y,\\
0 & \text{if}\quad x<y.
\end{array}
\right.  \label{eq:02}%
\end{equation}
Then
\[
\alpha(x,y)\geq1\Leftrightarrow x\geq y\text{\quad}(x,y\in X),
\]
hence a mapping $T:X\rightarrow X$ is $\alpha$--admissible \emph{iff} it is nondecreasing.
\end{example}

\begin{lemma}
\label{Lemma:2}Assume that $T$ is $\alpha$--admissible and $\alpha
$--contractive of Meir--Keeler type. Let $x,y\in X$ such that $\alpha
(x,y)\geq1$. Then%
\begin{equation}
\alpha(T^{n}x,T^{n}y)\geq1\quad\text{for all }n\in\mathbb{N}\text{,}
\label{eq:03}%
\end{equation}
the sequence $\left\{  d(T^{n}x,T^{n}y)\right\}  $ is nonincreasing, and%
\[
d(T^{n}x,T^{n}y)\rightarrow0\quad\text{(as }n\rightarrow\infty\text{).}%
\]

\end{lemma}

\begin{proof}
Since $T$ is $\alpha$--admissible and $\alpha(x,y)\geq1$, then (\ref{eq:03})
follows simply by induction on $n$.

Next, let $n\in\mathbb{N}$. If $T^{n}x\neq T^{n}y$, then, by (\ref{eq:03}) and
Lemma \ref{Lemma:1}, it follows that%
\[
d(T^{n+1}x,T^{n+1}y)\leq\alpha(T^{n}x,T^{n}y)d(T^{n+1}x,T^{n+1}y)=\alpha
(T^{n}x,T^{n}y)d(T(T^{n}x),T(T^{n}y))<d(T^{n}x,T^{n}y).
\]
Else, if $T^{n}x=T^{n}y$, then $d(T^{n+1}x,T^{n+1}y)=d(T^{n}x,T^{n}y)$.
Concluding, $\left\{  d(T^{n}x,T^{n}y)\right\}  $ is nonincreasing, hence
convergent to some $\varepsilon\geq0$.

Assume that $\varepsilon>0$, and let $p\in\mathbb{N}$ such that $\varepsilon
\leq d(T^{p}x,T^{p}y)<\varepsilon+\delta(\varepsilon)$. Then $\alpha
(T^{p}x,T^{p}y)d(T(T^{p}x),T(T^{p}y))<\varepsilon$, and further, by
(\ref{eq:03}), we get $d(T^{p+1}x,T^{p+1}y)<\varepsilon$, which is clearly not
possible, hence our assumption on $\varepsilon$ is wrong. Concluding, we have
necessarily $\varepsilon=0$.
\end{proof}

\begin{definition}
We say that a sequence $\{x_{n}\}$ in $X$ is $(T,\alpha)$--orbital if
$x_{n}=T^{n}x_{0}$ and $\alpha(x_{n},x_{n+1})\geq1$ for all $n\in\mathbb{N}$.
\end{definition}

\begin{definition}
We say that $T$ is $\alpha$--orbitally continuous if for every $(T,\alpha
)$--orbital sequence $\{x_{n}\}$ in $X$ such that $x_{n}\rightarrow x\in X$ as
$n\rightarrow+\infty$, there exists a subsequence $\{x_{n(k)}\}$ of
$\{x_{n}\}$ such that $Tx_{n(k)}\rightarrow Tx$ as $k\rightarrow+\infty$.
\end{definition}

\begin{remark}
Clearly, if $T$ is continuous, then $T$ is $\alpha$--orbitally continuous (for
any $\alpha$).
\end{remark}

\begin{definition}
We say that $(X,d)$ is $(T,\alpha)$--regular if for every $(T,\alpha
)$--orbital sequence $\{x_{n}\}$ in $X$ such that $x_{n}\rightarrow x\in X$ as
$n\rightarrow+\infty$, there exists a subsequence $\{x_{n(k)}\}$ of
$\{x_{n}\}$ such that $\alpha(x_{n(k)},x)\geq1$ for all $k$.
\end{definition}

\begin{definition}
We say that $(X,d)$ is $\alpha$--regular if for every sequence $\{x_{n}\}$ in
$X$ such that $x_{n}\rightarrow x\in X$ as $n\rightarrow+\infty$ and
$\alpha(x_{n},x_{n+1})\geq1$ for all $n$, there exists a subsequence
$\{x_{n(k)}\}$ of $\{x_{n}\}$ such that $\alpha(x_{n(k)},x)\geq1$ for all $k$.
\end{definition}

\begin{remark}
Clearly, if $(X,d)$ is $\alpha$--regular, then it is also $(T,\alpha
)$--regular (for any $T$).
\end{remark}

\begin{example}
Let $d$ be the usual (Euclidian) distance on $\mathbb{R}$, and $\alpha
:\mathbb{R}\times\mathbb{R}\rightarrow\lbrack0,+\infty)$ given by
(\ref{eq:02}). Then $(\mathbb{R},d)$ is $\alpha$--regular.
\end{example}

\begin{definition}
Let $N\in\mathbb{N}$. We say that $\alpha$ is $N$--transitive (on $X$) if
\[
x_{0},x_{1},\dots,x_{N+1}\in X:\alpha(x_{i},x_{i+1})\geq1\text{ for all }%
i\in\{0,1,\ldots,N\}\Longrightarrow\alpha(x_{0},x_{N+1})\geq1.
\]
In particular, we say that $\alpha$ is transitive if it is $1$--transitive,
i.e.,%
\[
x,y,z\in X:\alpha(x,y)\geq1,~\alpha(y,z)\geq1\Longrightarrow\alpha(x,z)\geq1.
\]

\end{definition}

The following remarks are immediate consequences of the previous definition.

\begin{remark}
Any function $\alpha:X\times X\rightarrow\lbrack0,+\infty)$ is $0$-transitive.
\end{remark}

\begin{remark}
If $\alpha$ is $N$ transitive, then it is $kN$--transitive for all
$k\in\mathbb{N}$.
\end{remark}

\begin{remark}
If $\alpha$ is transitive, then it is $N$--transitive for all $N\in\mathbb{N}$.
\end{remark}

\begin{example}
Let $X=\mathbb{R}$. Then $\alpha$ defined by (\ref{eq:02}) is transitive.
\end{example}

\begin{example}
\label{Ex:1}Let $N\in\mathbb{N}\setminus\{0\}$ and $\left\{  A_{1}%
,\ldots,A_{N}\right\}  $ a family of nonempty sets. Let $X=\bigcup_{i=1}%
^{N}A_{i}$ and $R=\bigcup_{i=1}^{N}\left(  A_{i}\times A_{i+1}\right)  $ (with
$A_{N+1}:=A_{1}$). Define $\alpha:X\times X\rightarrow\lbrack0,+\infty)$ by
\[
\alpha(x,y)=\left\{
\begin{array}
[c]{ll}%
1, & \text{if }(x,y)\in R\\
0, & \text{otherwise.}%
\end{array}
\right.
\]
Then $\alpha$ is $N$--transitive, but not necessarily transitive (see, also,
Corollary \ref{Cor:7}).
\end{example}

\begin{definition}
Let $x,y\in X$. A vector $\zeta=(z_{0},z_{1},\ldots,z_{n})\in X^{n+1}$ is
called an $\alpha$--chain (of order $n$) from $x$ to $y$ if $z_{0}=x$,
$z_{n}=y$ and, for every $i\in\{1,2,\ldots,n\}$,
\[
\alpha(z_{i-1},z_{i})\geq1\text{ or }\alpha(z_{i},z_{i-1})\geq1\text{.}%
\]

\end{definition}

\begin{definition}
We say that $X$ is $\alpha$--connected if for every $x,y\in X$ with $x\neq y$,
there exists an $\alpha$--chain from $x$ to $y$.
\end{definition}

%%% -----------------------------------------------------------------------

\subsection{Existence and uniqueness of fixed points}

Now, we are ready to present and prove the first main result of the paper.

\begin{theorem}
\label{Th:1}Let $(X,d)$ be a complete metric space, $\alpha:X\times
X\rightarrow\lbrack0,+\infty)$ a $N$--transitive mapping (for some
$N\in\mathbb{N}\setminus\{0\}$) and $T:X\rightarrow X$ an $\alpha
$--contractive mapping of Meir--Keeler type satisfying the following conditions:

\begin{enumerate}
\item[\textrm{(A1)}] $T$ is $\alpha$--admissible;

\item[\textrm{(A2)}] there exists $x_{0}\in X$ such that $\alpha(x_{0}%
,Tx_{0})\geq1$;

\item[\textrm{(A3)}] $T$ is $\alpha$--orbitally continuous.
\end{enumerate}

Then $T$ has a fixed point, that is, there exists $x^{\ast}\in X$ such that
$Tx^{\ast}=x^{\ast}$.
\end{theorem}

\begin{proof}
Define the sequence $\{x_{n}\}$ in $X$ by $x_{n+1}=Tx_{n}$ for all
$n\in\mathbb{N}$; equivalently, $x_{n}=T^{n}x_{0}$. Since $\alpha(x_{0}%
,Tx_{0})\geq1$, then by Lemma \ref{Lemma:2} we get
\begin{equation}
\alpha(x_{n},x_{n+1})\geq1\quad\text{for all }n\in\mathbb{N} \label{eq:04a}%
\end{equation}
and
\begin{equation}
d(x_{n},x_{n+1})\rightarrow0\quad\text{as}\quad n\rightarrow+\infty.
\label{eq:04}%
\end{equation}

Fix $\varepsilon>0$. Without any loss of generality, we may assume that
$\delta(\varepsilon)\leq\varepsilon$. Using (\ref{eq:04}), there exists $k$
such that
\begin{equation}
d(x_{n},x_{n+1})<\frac{\delta(\varepsilon)}{N}\quad\text{for all }n\geq
k\text{.} \label{eq:05}%
\end{equation}
We introduce the set $Y\subset X$ defined by
\[
Y:=\left\{  x\in X:\text{ there exists }q(x)\in\{0,1,\ldots,N-1\}~~\text{such
that}~~d(x_{k+q(x)},x)<\varepsilon+\delta(\varepsilon)\text{ and }%
\alpha(x_{k+q(x)},x)\geq1\right\}  .
\]

Fix $x\in Y$. Our first claim is that
\begin{equation}
T^{N}x\in Y\text{ and }q\left(  T^{N}x\right)  =q(x). \label{eq:05y}%
\end{equation}
For short, let $q:=q(x)$.

First, we prove that
\begin{equation}
d(x_{k+q},T^{N}x)<\varepsilon+\delta(\varepsilon). \label{eq:07}%
\end{equation}
Using the triangle inequality and (\ref{eq:05}), we obtain%
\[
d(x_{k+q},T^{N}x)\leq\sum_{i=0}^{N-1}d(x_{k+q+i},x_{k+q+i+1})+d(x_{k+q+N}%
,T^{N}x)<\delta(\varepsilon)+d(T^{N}x_{k+q},T^{N}x),
\]
while $\alpha(x_{k+q},x)\geq1$ leads to%
\[
d(T^{N}x_{k+q},T^{N}x)\leq d(Tx_{k+q},Tx)\leq d(x_{k+q},x)
\]
by Lemma \ref{Lemma:2}; hence, we conclude that%
\begin{equation}
d(x_{k+q},T^{N}x)<d(Tx_{k+q},Tx)+\delta(\varepsilon)\leq d(x_{k+q}%
,x)+\delta(\varepsilon). \label{eq:06}%
\end{equation}
Clearly, if $d(x_{k+q},x)<\varepsilon$, then (\ref{eq:06}) leads to
(\ref{eq:07}), so it is enough to consider the case when $\varepsilon\leq
d(x_{k+q},x)$. Then $x\in Y$ leads to $\varepsilon\leq d(x_{k+q}%
,x)<\varepsilon+\delta(\varepsilon)$. Using next that $T$ is an $\alpha
$--contractive mapping of Meir--Keeler type, we obtain that $\alpha
(x_{k+q},x)d(Tx_{k+q},Tx)<\varepsilon$, and since $\alpha(x_{k+q},x)\geq1$, we
arrive to
\begin{equation}
d(Tx_{k+q},Tx)<\varepsilon; \label{eq:08}%
\end{equation}
hence (\ref{eq:07}) follows again by (\ref{eq:06}) and (\ref{eq:08}).

Next, we prove that%
\begin{equation}
\alpha(x_{k+q},T^{N}x)\geq1. \label{eq:09}%
\end{equation}
Indeed,
\begin{equation}
\alpha(x_{k+q+i},x_{k+q+i+1})\geq1\text{\quad for all }i\in\{0,1,\ldots,N-1\}
\label{eq:10}%
\end{equation}
by (\ref{eq:04a}). Also, $\alpha(x_{k+q},x)\geq1$ leads by Lemma \ref{Lemma:2}
to
\begin{equation}
\alpha(x_{k+q+N},T^{N}x)\geq1. \label{eq:11}%
\end{equation}
Now, using (\ref{eq:10}), (\ref{eq:11}) and the $N$--transitivity of $\alpha$,
we finally get (\ref{eq:09}).

Concluding, our first claim (\ref{eq:05y}) is proven.

Our second claim is
\begin{equation}
x_{k+i+1}\in Y\text{ and }q(x_{k+i+1})=i\quad\text{for all }i\in
\{0,1,\ldots,N-1\}\text{.} \label{eq:12}%
\end{equation}
Indeed, $d(x_{k+i},x_{k+i+1})<\frac{\delta(\varepsilon)}{N}<\varepsilon
+\delta(\varepsilon)$ by (\ref{eq:05}), while $\alpha(x_{k+i},x_{k+i+1})\geq1$
by (\ref{eq:04a}), which proves (\ref{eq:12}).

Now, by (\ref{eq:05y}) and (\ref{eq:12}), we can easily conclude that
\begin{equation}
x_{n}\in Y\text{ and }q(x_{n})=(n-k-1)\operatorname{mod}N\quad\text{for all
}n\geq k+1. \label{eq:12a}%
\end{equation}

Finally, let $m,n\geq k+1$ and assume that $q(x_{n})\leq q(x_{m})$ without any
loss of generality. Then, by the triangle inequality, (\ref{eq:05}) and
(\ref{eq:12a}), it follows that
\begin{align*}
d(x_{n},x_{m})  &  \leq d(x_{n},x_{k+q(x_{n})})+\sum_{i=q(x_{n})}^{q(x_{m}%
)-1}d\left(  x_{k+i},x_{k+i+1}\right)  +d(x_{k+q(x_{m})},x_{m})\\
&  <2(\varepsilon+\delta(\varepsilon))+(q(x_{m})-q(x_{n}))\frac{\delta
(\varepsilon)}{N}\leq2(\varepsilon+\delta(\varepsilon))+\delta(\varepsilon
)\leq5\varepsilon.
\end{align*}

Concluding, $\{x_{n}\}$ is a Cauchy sequence in the complete metric space
$(X,d)$, hence convergent to some $x^{\ast}\in X$. Moreover, $\{x_{n}\}$ is a
$(T,\alpha)$--orbital sequence by (\ref{eq:04a}), hence, by (A3), there exists
a subsequence $\{x_{n(k)}\}$ of $\{x_{n}\}$ such that $Tx_{n(k)}\rightarrow
Tx^{\ast}$ as $k\rightarrow+\infty$. But $Tx_{n(k)}=x_{n(k)+1}\rightarrow
x^{\ast}$ as $k\rightarrow+\infty$, hence $Tx^{\ast}=x^{\ast}$ by the
uniqueness of the limit, which concludes the proof.
\end{proof}

In the next theorem, we replace the continuity of the mapping $T$ by a
regularity condition over the metric space $(X,d)$.

\begin{theorem}
\label{Th:2}In the conditions of Theorem \ref{Th:1}, if (A3) is replaced with:

\begin{enumerate}
\item[\textrm{(A4)}] $(X,d)$ is $(T,\alpha)$--regular,
\end{enumerate}

\noindent then the conclusion of Theorem \ref{Th:1} holds.
\end{theorem}

\begin{proof}
Following the proof of Theorem \ref{Th:1}, we only have to prove that
$x^{\ast}$ is a fixed point of $T$. Since $\{x_{n}\}$ is a $(T,\alpha
)$--orbital sequence, then, by (A4), there exists a subsequence $\{x_{n(k)}\}$
of $\{x_{n}\}$ such that
\[
\alpha(x_{n(k)},x^{\ast})\geq1\quad\text{for all}\quad k\in\mathbb{N}.
\]
Next, using Lemma \ref{Lemma:1}, we get
\[
d(Tx_{n(k)},Tx^{\ast})\leq\alpha(x_{n(k)},x^{\ast})d(Tx_{n(k)},Tx^{\ast})\leq
d(x_{n(k)},x^{\ast})\quad\text{for all}\quad k\in\mathbb{N}%
\]
(with equality when $x_{n(k)}=x^{\ast}$). As $x_{n(k)}\rightarrow x^{\ast}$,
we obtain that $x_{n(k)+1}=Tx_{n(k)}\rightarrow Tx^{\ast}$. As $\{x_{n(k)+1}%
\}$ is a subsequence of $\{x_{n}\}$ and $x_{n}\rightarrow x^{\ast}$ we have
$x_{n(k)+1}\rightarrow x^{\ast}$. Now, the uniqueness of the limit gives us
$Tx^{\ast}=x^{\ast}$ and the proof is complete.
\end{proof}

To assure the uniqueness of the fixed point, we will consider the following
additional assumption.

\begin{enumerate}
\item[\textrm{(A5)}] $X$ is $\alpha$--connected.
\end{enumerate}

This is the purpose of the next theorem.

\begin{theorem}
\label{Th:3}If adding (A5) to the hypotheses of Theorem \ref{Th:1} (or Theorem
\ref{Th:2}), then $x^{\ast}$ is the unique fixed point of $T$ and
$T^{n}(x)\rightarrow x^{\ast}$ (as $n\rightarrow\infty$) for every $x\in X$.
\end{theorem}

\begin{proof}
Let $x\in X\setminus\{x^{\ast}\}$. By (A5), there exists $(x^{\ast}%
=z_{0},z_{1},\ldots,z_{n}=x)$ an $\alpha$--chain from $x^{\ast}$ to $x$.
Since
\[
\alpha(z_{i-1},z_{i})\geq1\text{ or }\alpha(z_{i},z_{i-1})\geq1\quad\text{for
all }i\in\{1,2,\ldots,n\},
\]
it follows by Lemma \ref{Lemma:2} and the symmetry of $d$, that
\begin{equation}
d(T^{n}(z_{i-1}),T^{n}(z_{i}))\rightarrow0\text{ (as }n\rightarrow
+\infty\text{)\quad for all }i\in\{1,2,\ldots,n\}\text{.} \label{eq:13}%
\end{equation}
Now, since $z_{0}=x^{\ast}$ is a fixed point of $T$, it follows that
$T^{n}(z_{0})=x^{\ast}$ for all $n$, which finally leads to
\[
T^{n}z_{i}\rightarrow x^{\ast}\text{ (as }n\rightarrow+\infty\text{)\quad for
all }i\in\{1,2,\ldots,n\},
\]
using (\ref{eq:13}); hence, $T^{n}x\rightarrow x^{\ast}$ (as $n\rightarrow
+\infty$). In particular, if $x$ is another fixed point of $T$, it follows
that $x=x^{\ast}$ which is a contradiction, and the proof is concluded.
\end{proof}

%%%%%%%%%%%%%%%%%%%%%%%%%%%%%%%%%%%%%%%%%%%%%%%%%%%%%

\section{Some corollaries\label{Sec:3}}

In this section, we will derive some corollaries from our previous theorems.

\subsection{Coupled fixed point theorems for bivariate $\alpha$--contractive
mappings of Meir--Keeler type on complete metric spaces}

The theorems obtained in the previous section allow us to derive some coupled
fixed point results in complete metric spaces. First, let us recall the
following definitions.

\begin{definition}
[\cite{Bhaskar2006}]Let $X$ be a nonempty set and $F:X\times X\rightarrow X$
be a given mapping. A pair $(x,y)\in X\times X$ is called a coupled fixed
point of $F$ if $F(x,y)=x$ and $F(y,x)=y$.

Also, $x\in X$ is called a fixed point of $F$ if $(x,x)$ is a coupled fixed
point, i.e., $F(x,x)=x$.
\end{definition}

\begin{definition}
[\cite{Rus2011}]Let $X$ be a nonempty set, and $F,G:X\times X\rightarrow X$.
The symmetric composition (or, the $s$\emph{-composition} for short) of $A$
and $B$ is defined by%
\[
G\ast F:X\times X\rightarrow X,\quad(G\ast F)(x,y)=G(F(x,y),F(y,x))\quad
(x,y\in X).
\]

\end{definition}

\begin{remark}
[\cite{Rus2011}]The $s$-composition is an associative law. Also, the
projection mapping%
\[
P_{X}:X\times X\rightarrow X,\quad P(x,y)=x\quad(x,y\in X)
\]
is the identity element with respect to the $s$-composition (i.e., $F\ast
P_{X}=P_{X}\ast F=F$ for all $F:X\times X\rightarrow X$). Consequently, for
any $F:X\times X\rightarrow X$ one can define the functional powers (i.e., the
iterates) of $F$ with respect to the $s$-composition by%
\[
F^{n+1}=F\ast F^{n}=F^{n}\ast F\quad(n\in\mathbb{N}),\quad F^{0}=P_{X}\text{.}%
\]

\end{remark}

We have the following result.

\begin{corollary}
\label{Cor:1}Let $(X,d)$ be a complete metric space, $\alpha:(X\times
X)\times(X\times X)\rightarrow\lbrack0,+\infty)$ a $N$--transitive mapping on
$X\times X$ for some $N\in\mathbb{N}\setminus\{0\}$, and $F:X\times
X\rightarrow X$ such that for every $\varepsilon>0$ there exists
$\delta(\varepsilon)>0$ for which:
\begin{equation}
(x,y),(u,v)\in X\times X:\varepsilon\leq\frac{d(x,u)+d(y,v)}{2}<\varepsilon
+\delta(\varepsilon)\Rightarrow\alpha((x,y),(u,v))d(F(x,y),F(u,v))<\varepsilon
. \label{eq:14}%
\end{equation}

Suppose that

\begin{enumerate}
\item[\textrm{(B1)}] for all $(x,y),(u,v)\in X\times X$,
\[
\alpha((x,y),(u,v))\geq1\Longrightarrow\alpha
((F(x,y),F(y,x)),(F(u,v),F(v,u)))\geq1;
\]

\item[\textrm{(B2)}] there exists $(x_{0},y_{0})\in X\times X$ such that
\[
\alpha\left(  (x_{0},y_{0}),(F(x_{0},y_{0}),F(y_{0},x_{0}))\right)  \geq
1\quad\text{and}\quad\alpha\left(  (F(y_{0},x_{0}),F(x_{0},y_{0}%
)),(y_{0},x_{0})\right)  \geq1;
\]

\item[\textrm{(B3)}] $F$ is continuous.
\end{enumerate}

Then $F$ has a coupled fixed point, that is, there exists $(x^{\ast},y^{\ast
})\in X\times X$ such that $x^{\ast}=F(x^{\ast},y^{\ast})$ and $y^{\ast
}=F(y^{\ast},x^{\ast})$.
\end{corollary}

\begin{proof}
Consider
\[
D\left(  (x,y),(u,v)\right)  :=\frac{1}{2}\left(  d(x,u)+d(y,v)\right)
\quad\text{for all }(x,y),(u,v)\in X\times X.
\]
Then, clearly, $(X\times X,D)$ is a complete metric space. Also, let
$T:X\times X\rightarrow X\times X$ be defined by%
\[
T(x,y)=(F(x,y),F(y,x))\quad\text{for all }(x,y)\in X\times X
\]
and $\beta:(X\times X)\times(X\times X)\rightarrow\lbrack0,+\infty)$ be given
by
\begin{equation}
\beta((x,y),(u,v))=\min\left\{  \alpha((x,y),(u,v)),\alpha
((v,u),(y,x))\right\}  \quad\text{for all }(x,y),(u,v)\in X\times X.
\label{eq:15}%
\end{equation}

First, we prove that $\beta$ is $N$-transitive. Let $(x_{i},y_{i})\in X\times
X$ ($i\in\{0,1,\ldots,N+1\}$) such that $\beta\left(  (x_{i},y_{i}%
),(x_{i+1},y_{i+1})\right)  \geq1$ for all $i\in\{0,1,\ldots,N\}$. By the
definition of $\beta$, it follows that
\[
\alpha\left(  (x_{i},y_{i}),(x_{i+1},y_{i+1})\right)  \geq1\text{ and }%
\alpha\left(  (y_{i+1},x_{i+1}),(y_{i},x_{i})\right)  \geq1\text{ \quad for
all }i\in\{0,1,\ldots,N\},
\]
hence, by the $N$--transitivity of $\alpha$, we have that
\[
\alpha\left(  (x_{0},y_{0}),(x_{N+1},y_{N+1})\right)  \geq1\text{ and }%
\alpha\left(  (y_{N+1},x_{N+1}),(y_{0},x_{0})\right)  \geq1\text{,}%
\]
which concludes our argument.

We claim next that $T$ is a $\beta$--contractive mapping of Meir--Keeler type
(with respect to $D$). Indeed, let $\varepsilon>0$ and let $\delta
(\varepsilon)>0$ for which (\ref{eq:14}) is satisfied. If $(x,y),(u,v)\in
X\times X$ are such that $\varepsilon\leq D\left(  (x,y),(u,v)\right)
<\varepsilon+\delta(\varepsilon)$, then also $\varepsilon\leq D\left(
(v,u),(y,x)\right)  <\varepsilon+\delta(\varepsilon)$ by the definition of
$D$, hence%
\[%
\begin{array}
[c]{c}%
\alpha((x,y),(u,v))d(F(x,y),F(u,v))<\varepsilon\\
\alpha((v,u),(y,x))d(F(v,u),F(y,x))<\varepsilon
\end{array}
\]
by (\ref{eq:14}). These two inequalities lead straight to
\[
\beta((x,y),(u,v))D\left(  T(x,y),T(u,v)\right)  <\varepsilon,
\]
which proves our claim.

Next, it is easy to check that $T$ is $\beta$--admissible by (B1). Moreover,
(B2) ensures that $\beta((x_{0},y_{0}),T(x_{0},y_{0}))\geq1$, while (B3)
ensures that $T$ is continuous, hence $\beta$--orbitally continuous.

Concluding, all the hypotheses of Theorem \ref{Th:1} applied to the metric
space $(X\times X,D)$, the mapping $T$ and the function $\beta$ are satisfied,
hence $T$ has a fixed point $(x^{\ast},y^{\ast})\in X\times X$, meaning that
$(x^{\ast},y^{\ast})$ is a coupled fixed point of $F$. The proof is now complete.
\end{proof}

\begin{corollary}
\label{Cor:2}In the conditions of Corollary \ref{Cor:1}, if (B3) is replaced with:

\begin{enumerate}
\item[\textrm{(B4)}] for every sequence $\{(x_{n},y_{n})\}$ in $X\times X$
such that $x_{n}\rightarrow x\in X$, $y_{n}\rightarrow y\in X$ as
$n\rightarrow+\infty$, and%
\[
\alpha((x_{n},y_{n}),(x_{n+1},y_{n+1}))\geq1,\quad\alpha((y_{n+1}%
,x_{n+1}),(y_{n},x_{n}))\geq1\quad\text{for all }n\in\mathbb{N},
\]
there exists a subsequence $\{(x_{n(k)},y_{n(k)})\}$ such that%
\[
\alpha((x_{n(k)},y_{n(k)}),(x,y))\geq1,\quad\alpha((y,x),(y_{n(k)}%
,x_{n(k)}))\geq1\quad\text{for all }k\in\mathbb{N};
\]

\end{enumerate}

\noindent then the conclusion of Corollary \ref{Cor:1} holds.
\end{corollary}

\begin{proof}
Using the notations in the proof of Corollary \ref{Cor:1}, it easily follows
by (B4) that $(X\times X,D)$ is $\beta$--regular, hence $(T,\beta)$--regular.
By following the proof of Corollary \ref{Cor:1}, the conclusion follows by
Theorem \ref{Th:2} applied to the metric space $(X\times X,D)$, the mapping
$T$ and the function $\beta$.
\end{proof}

For the uniqueness of the coupled fixed point, we consider the following assumption.

\begin{enumerate}
\item[\textrm{(B5)}] $X\times X$ is $\beta$--connected, where $\beta$ is
defined by (\ref{eq:15}).
\end{enumerate}

\begin{corollary}
\label{Cor:3} If adding condition (B5) to the hypotheses of Corollary
\ref{Cor:1} (or Corollary \ref{Cor:2}) then $x^{\ast}=y^{\ast}$, $(x^{\ast
},x^{\ast})$ is the unique coupled fixed point of $F$ and $x^{\ast}$ is the
unique fixed point of $F$. Moreover, $F^{n}(x,y)\rightarrow x^{\ast}$ as
$n\rightarrow\infty$ for all $x,y\in X$.
\end{corollary}

\begin{proof}
We use the notations in the proof of Corollary \ref{Cor:1}. Then, by Theorem
\ref{Th:3}, it follows that $(x^{\ast},y^{\ast})$ is the unique fixed point of
$T$, hence the unique coupled fixed point of $F$. Since $(y^{\ast},x^{\ast})$
is also a coupled fixed point of $F$, then $(x^{\ast},y^{\ast})=(y^{\ast
},x^{\ast})$, hence $x^{\ast}=y^{\ast}$, meaning also that $x^{\ast}$ is the
unique fixed point of $F$. Since $T^{n}(x,y)=\left(  F^{n}(x,y),F^{n}%
(y,x)\right)  $ for all $n\in\mathbb{N}$ and $x,y\in X$, the proof is complete.
\end{proof}

We conclude this subsection with a particular form of the above corollaries,
when $\alpha$ is represented as:%
\begin{equation}
\alpha\left(  (x,y),(u,v)\right)  =\min\left\{  \alpha_{0}(x,u),\alpha
_{0}(v,y)\right\}  \quad\left(  (x,y),(u,v)\in X\times X\right)  \text{,}
\label{eq:16}%
\end{equation}
where $\alpha_{0}:X\times X\rightarrow\lbrack0,+\infty)$. Note that, in this
case, $\beta=\alpha$. We subsume the conclusions of Corollaries \ref{Cor:1},
\ref{Cor:2} and \ref{Cor:3} in one single result, as follows:

\begin{corollary}
\label{Cor:4}Let $(X,d)$ be a complete metric space, $\alpha_{0}:X\times
X\rightarrow\lbrack0,+\infty)$ a $N$--transitive mapping on $X\times X$ for
some $N\in\mathbb{N}\setminus\{0\}$, and $F:X\times X\rightarrow X$ such that
for every $\varepsilon>0$ there exists $\delta(\varepsilon)>0$ for which:
\[
(x,y),(u,v)\in X\times X:\varepsilon\leq\frac{d(x,u)+d(y,v)}{2}<\varepsilon
+\delta(\varepsilon)\Rightarrow\min\left\{  \alpha_{0}(x,u),\alpha
_{0}(v,y)\right\}  d(F(x,y),F(u,v))<\varepsilon.
\]

Suppose that

\begin{enumerate}
\item[\textrm{(C1)}] for all $(x,y),(u,v)\in X\times X$,
\[
\alpha_{0}(x,u)\geq1,~\alpha_{0}(v,y)\geq1\Longrightarrow\alpha_{0}\left(
F(x,y),F(u,v)\right)  \geq1;
\]

\item[\textrm{(C2)}] there exists $(x_{0},y_{0})\in X\times X$ such that
\[
\alpha_{0}\left(  x_{0},F(x_{0},y_{0})\right)  \geq1,\quad\alpha_{0}\left(
F(y_{0},x_{0}),y_{0}\right)  \geq1.
\]
If either

\item[\textrm{(C3)}] $F$ is continuous,
\end{enumerate}

\noindent or

\begin{enumerate}
\item[\textrm{(C4)}] for every sequence $\{(x_{n},y_{n})\}$ in $X\times X$
such that $x_{n}\rightarrow x\in X$, $y_{n}\rightarrow y\in X$ as
$n\rightarrow+\infty$, and%
\[
\alpha_{0}(x_{n},x_{n+1})\geq1,\quad\alpha_{0}(y_{n+1},y_{n})\geq
1\quad\text{for all }n\in\mathbb{N},
\]
there exists a subsequence $\{(x_{n(k)},y_{n(k)})\}$ such that%
\[
\alpha_{0}\left(  x_{n(k)},x\right)  \geq1,\quad\alpha_{0}(y,y_{n(k)}%
)\geq1\quad\text{for all}\ k\in\mathbb{N};
\]

\end{enumerate}

\noindent then $F$ has a coupled fixed point, that is, there exists $(x^{\ast
},y^{\ast})\in X\times X$ such that $x^{\ast}=F(x^{\ast},y^{\ast})$ and
$y^{\ast}=F(y^{\ast},x^{\ast})$.

Additionally, if

\begin{enumerate}
\item[\textrm{(C5)}] $X$ is $\alpha_{0}$--connected,
\end{enumerate}

\noindent then $x^{\ast}=y^{\ast}$, $(x^{\ast},x^{\ast})$ is the unique
coupled fixed point of $F$, $x^{\ast}$ is the unique fixed point of $F$ and
$F^{n}(x,y)\rightarrow x^{\ast}$ as $n\rightarrow\infty$ for all $x,y\in X$.
\end{corollary}

\begin{proof}
It checks easily that the hypotheses of Corollaries \ref{Cor:1}, \ref{Cor:2}
and \ref{Cor:3} are satisfied, with $\alpha$ defined by (\ref{eq:16}).
\end{proof}

%%%%%%%%%%%%%%%%%%%%%%%%%%%%%%%%%%%%%%%

\subsection{Fixed point theorems for ${\mathcal{R}}$--contractive mappings of
Meir--Keeler type on a metric space endowed with a $N$--transitive binary
relation}

The notions and results in Section \ref{Sec:2} easily translate to the setting
of metric spaces endowed with a $N$--transitive binary relation.

In what follows, let $(X,d)$ be a metric space, ${\mathcal{R}}$ be a binary
relation over $X$ and $T:X\rightarrow X$. We first start with some terminology
that is symmetrical to that in Section \ref{Sec:2}.

\begin{definition}
We say that $T$ is a ${\mathcal{R}}$--contractive mapping of Meir--Keeler type
(with respect to $d$) if for all $\varepsilon>0$, there exists $\delta
(\varepsilon)>0$ such that
\[
x,y\in X:x{\mathcal{R}}y,~\varepsilon\leq d(x,y)<\varepsilon+\delta
(\varepsilon)\Rightarrow d(Tx,Ty)<\varepsilon.
\]

\end{definition}

\begin{definition}
We say that $T$ is ${\mathcal{R}}$--preserving if
\[
x,y\in X:x{\mathcal{R}}y\Rightarrow Tx{\mathcal{R}}Ty.
\]

\end{definition}

\begin{definition}
We say that a sequence $\{x_{n}\}$ in $X$ is $(T,{\mathcal{R}})$--orbital if
$x_{n}=T^{n}x_{0}$ and $x_{n}{\mathcal{R}}x_{n+1}$ for all $n\in\mathbb{N}$.
\end{definition}

\begin{definition}
We say that $T$ is ${\mathcal{R}}$--orbitally continuous if for every
$(T,{\mathcal{R}})$--orbital sequence $\{x_{n}\}$ in $X$ such that
$x_{n}\rightarrow x\in X$ as $n\rightarrow+\infty$, there exists a subsequence
$\{x_{n(k)}\}$ of $\{x_{n}\}$ such that $Tx_{n(k)}\rightarrow Tx$ as
$k\rightarrow+\infty$.
\end{definition}

\begin{remark}
Clearly, if $T$ is continuous, then $T$ is ${\mathcal{R}}$--orbitally
continuous (for any ${\mathcal{R}}$).
\end{remark}

\begin{definition}
We say that $(X,d)$ is $(T,{\mathcal{R}})$--regular if for every
$(T,{\mathcal{R}})$--orbital sequence $\{x_{n}\}$ in $X$ such that
$x_{n}\rightarrow x\in X$ as $n\rightarrow+\infty$, there exists a subsequence
$\{x_{n(k)}\}$ of $\{x_{n}\}$ such that $x_{n(k)}{\mathcal{R}}x$ for all $k$.
\end{definition}

\begin{definition}
We say that $(X,d)$ is ${\mathcal{R}}$--regular if for every sequence
$\{x_{n}\}$ in $X$ such that $x_{n}\rightarrow x\in X$ as $n\rightarrow
+\infty$ and $x_{n}{\mathcal{R}}x_{n+1}$ for all $n$, there exists a
subsequence $\{x_{n(k)}\}$ of $\{x_{n}\}$ such that $x_{n(k)}{\mathcal{R}}x$
for all $k$.
\end{definition}

\begin{remark}
Clearly, if $(X,d)$ is ${\mathcal{R}}$--regular, then it is also
$(T,{\mathcal{R}})$--regular (for any $T$).
\end{remark}

\begin{definition}
Let $N\in\mathbb{N}$. We say that ${\mathcal{R}}$ is $N$--transitive (on $X$)
if
\[
x_{0},x_{1},x_{2},\dots,x_{N},x_{N+1}\in X:x_{i}{\mathcal{R}}x_{i+1}\text{ for
all }i\in\{0,1,\ldots,N\}\Longrightarrow x_{0}{\mathcal{R}}x_{N+1}.
\]
In particular, for $N=1$ we recover the usual transitivity property.
\end{definition}

\begin{definition}
Let $x,y\in X$. A vector $\zeta=(z_{0},z_{1},\ldots,z_{n})\in X^{n+1}$ is
called a ${\mathcal{R}}$--chain (of order $n$) from $x$ to $y$ if $z_{0}=x$,
$z_{n}=y$ and
\[
z_{i-1}{\mathcal{R}}z_{i}\text{ or }z_{i}{\mathcal{R}}z_{i-1}\text{\quad for
every }i\in\{1,2,\ldots,n\}.
\]

\end{definition}

\begin{definition}
We say that $X$ is ${\mathcal{R}}$--connected if for every $x,y\in X$ with
$x\neq y$, there exists a ${\mathcal{R}}$--chain from $x$ to $y$.
\end{definition}

The main results in Section \ref{Sec:2} translate to the setting of metric
spaces endowed with an arbitrary binary relation as follows.

\begin{corollary}
\label{Cor:5}Let $(X,d)$ be a complete metric space, ${\mathcal{R}}$ a
$N$--transitive binary relation over $X$ (for some $N\in\mathbb{N}%
\setminus\{0\}$) and $T:X\rightarrow X$ a ${\mathcal{R}}$--contractive mapping
of Meir--Keeler type. Assume that:

\begin{enumerate}
\item[\textrm{(D1)}] $T$ is ${\mathcal{R}}$-preserving;

\item[\textrm{(D2)}] there exists $x_{0}\in X$ such that $x_{0}{\mathcal{R}%
}Tx_{0}$.
\end{enumerate}

If either

\begin{enumerate}
\item[\textrm{(D3)}] $T$ is continuous,
\end{enumerate}

\noindent or

\begin{enumerate}
\item[\textrm{(D4)}] $(X,d)$ is $(T,{\mathcal{R}})$--regular,
\end{enumerate}

\noindent then $T$ has a fixed point $x^{\ast}\in X$. Additionally, if

\begin{enumerate}
\item[\textrm{(D5)}] $X$ is ${\mathcal{R}}$--connected,
\end{enumerate}

\noindent then $x^{\ast}$ is the unique fixed point of $T$ and $T^{n}%
(x)\rightarrow x^{\ast}$ (as $n\rightarrow\infty$) for every $x\in X$.
\end{corollary}

\begin{proof}
Define the mapping $\alpha:X\times X\rightarrow\lbrack0,+\infty)$ by
\[
\alpha(x,y)=\left\{
\begin{array}
[c]{ll}%
1, & \text{if }x\mathcal{R}y\\
0, & \text{otherwise.}%
\end{array}
\right.
\]
The conclusions then follows directly from Theorems \ref{Th:1}, \ref{Th:2} and
\ref{Th:3}.
\end{proof}

The following result is a consequence of Corollary \ref{Cor:4} for bivariate
${\mathcal{R}}$--contractive mappings of Meir--Keeler type.

\begin{corollary}
\label{Cor:6}Let $(X,d)$ be a complete metric space, ${\mathcal{R}}$ a
$N$--transitive binary relation over $X$ (for some $N\in\mathbb{N}%
\setminus\{0\}$), and $F:X\times X\rightarrow X$ such that for every
$\varepsilon>0$ there exists $\delta(\varepsilon)>0$ for which:
\[
x,y,u,v\in X:x{\mathcal{R}}y,~v{\mathcal{R}}u,~\varepsilon\leq\frac
{d(x,u)+d(y,v)}{2}<\varepsilon+\delta(\varepsilon)\Rightarrow
d(F(x,y),F(u,v))<\varepsilon.
\]

Suppose that

\begin{enumerate}
\item[\textrm{(E1)}] for all $x,y,u,v\in X$,
\[
x{\mathcal{R}}y,~v{\mathcal{R}}u\Longrightarrow F(x,y){\mathcal{R}}F(u,v);
\]

\item[\textrm{(E2)}] there exists $(x_{0},y_{0})\in X\times X$ such that
\[
x_{0}{\mathcal{R}}F(x_{0},y_{0}),\quad F(y_{0},x_{0}){\mathcal{R}}y_{0}.
\]
If either

\item[\textrm{(E3)}] $F$ is continuous,
\end{enumerate}

\noindent or

\begin{enumerate}
\item[\textrm{(E4)}] for every sequence $\{(x_{n},y_{n})\}$ in $X\times X$
such that $x_{n}\rightarrow x\in X$, $y_{n}\rightarrow y\in X$ as
$n\rightarrow+\infty$, and $x_{n}{\mathcal{R}}x_{n+1},~y_{n+1}{\mathcal{R}%
}y_{n}$ for all $n\in\mathbb{N}$, there exists a subsequence $\{(x_{n(k)}%
,y_{n(k)})\}$ such that $x_{n(k)}{\mathcal{R}}x,~y{\mathcal{R}}y_{n(k)}$ for
all $k\in\mathbb{N}$,
\end{enumerate}

\noindent then $F$ has a coupled fixed point $(x^{\ast},y^{\ast})\in X\times
X$. Additionally, if

\begin{enumerate}
\item[\textrm{(E5)}] $X$ is ${\mathcal{R}}$--connected,
\end{enumerate}

then $x^{\ast}=y^{\ast}$, $(x^{\ast},x^{\ast})$ is the unique coupled fixed
point of $F$, $x^{\ast}$ is the unique fixed point of $F$ and $F^{n}%
(x,y)\rightarrow x^{\ast}$ as $n\rightarrow\infty$ for all $x,y\in X$.
\end{corollary}

\begin{proof}
Define the mapping $\alpha_{0}:X\times X\rightarrow\lbrack0,+\infty)$ by
\[
\alpha_{0}(x,y)=\left\{
\begin{array}
[c]{ll}%
1, & \text{if }x\mathcal{R}y\\
0, & \text{otherwise.}%
\end{array}
\right.
\]
The conclusions then follows directly from Corollary \ref{Cor:4}.
\end{proof}

\subsection{Fixed point results for cyclic contractive mappings of
Meir--Keeler type}

In this section, we obtain some fixed point results for cyclic $\alpha
$--contractions of Meir--Keeler type. We start by recalling the result
obtained by Kirk, Srinivasan and Veeramani in \cite{Kirk2003} for cyclic
contractive mappings.

\begin{theorem}
[\cite{Kirk2003}]Let $(X,d)$ be a complete metric space, $\left\{  A_{1}%
,A_{2},\ldots,A_{N}\right\}  $ a family of nonempty and closed subsets of $X$
and $T:X\rightarrow X$. Suppose that the following conditions hold:

\begin{enumerate}
\item[\textrm{(F1)}] $T(A_{i})\subseteq A_{i+1}$ for all $i\in\{1,2\dots,N\}$
(where $A_{N+1}=A_{1}$);

\item[\textrm{(F2)}] there exists $k\in(0,1)$ such that%
\[
d(Tx,Ty)\leq kd(x,y)\text{\quad for all }x\in A_{i},y\in A_{i+1}%
,i\in\{1,2\dots,N\}\text{.}%
\]
Then $\bigcap_{i=1}^{N}A_{i}$ is non-empty and $T$ has a unique fixed point in
$\bigcap_{i=1}^{N}A_{i}$.
\end{enumerate}
\end{theorem}

The aim of our next result is to weaken the contraction condition (F2) by
considering the following condition of Meir--Keeler type:

\begin{enumerate}
\item[(F3)] for every $\varepsilon>0$, there exists $\delta(\varepsilon)>0$
such that%
\[
x\in A_{i},y\in A_{i+1},i\in\{1,2,\ldots,N\}:\varepsilon\leq
d(x,y)<\varepsilon+\delta(\varepsilon))\Rightarrow d(Tx,Ty)<\varepsilon.
\]

\end{enumerate}

\begin{corollary}
\label{Cor:7}Let $(X,d)$ be a complete metric space, $\left\{  A_{1}%
,A_{2},\ldots,A_{N}\right\}  $ a family of nonempty and closed subsets of $X$
and $T:X\rightarrow X$. Suppose that \textrm{(F1)} and \textrm{(F3)} hold.

Then $\bigcap_{i=1}^{N}A_{i}$ is non-empty and $T$ has a fixed point $x^{\ast
}\in\bigcap_{i=1}^{N}A_{i}$. Moreover, $x^{\ast}$ is the unique fixed point of
$T$ in $\bigcup_{i=1}^{N}A_{i}$ and $T^{n}(x)\rightarrow x^{\ast}$ for all
$x\in\bigcup_{i=1}^{N}A_{i}$.
\end{corollary}

\begin{proof}
Let $Y:=\bigcup_{i=1}^{N}A_{i}$. Then $Y$ is a closed part of $X$; hence,
$(Y,d)$ is a complete metric space. Moreover, the restriction $\left.
T\right\vert _{Y}$ of $T$ to $Y$ is a self-map of $Y$, by (F1); for
convenience, we write $T$ instead of $\left.  T\right\vert _{Y}$.

Define the mapping $\alpha:Y\times Y\rightarrow\lbrack0,+\infty)$ by
\[
\alpha(x,y)=\left\{
\begin{array}
[c]{ll}%
1, & \text{if }(x,y)\in R:=\bigcup_{i=1}^{N}\left(  A_{i}\times A_{i+1}\right)
\\
0, & \text{otherwise.}%
\end{array}
\right.
\]
We check that the conditions in Theorem \ref{Th:2} are satisfied for the
complete metric space $(Y,d)$, the mappings $\alpha$ and $T$.

First, define $A_{i+kN}:=A_{i}$ for all $i\in\{1,2,\ldots,N\}$ and
$k\in\mathbb{Z}$. Then (F1) extends to%
\[
T(A_{i})\subseteq A_{i+1}\text{\quad for all }i\in\mathbb{Z}\text{.}%
\]

We check that $\alpha$ is $N$--transitive (see also Example \ref{Ex:1}).
Indeed, let $x_{0},x_{1},\dots,x_{N+1}\in Y$ such that $\alpha(x_{k}%
,x_{k+1})\geq1$ (i.e., $(x_{k},x_{k+1})\in R$) for all $k\in\{0,1,\ldots,N\}$.
This means that there exists $i\in\{1,\ldots,N\}$ such that
\[
x_{0}\in A_{i},~x_{1}\in A_{i+1},\ldots,x_{k}\in A_{i+k},\ldots,x_{N+1}\in
A_{i+N+1}=A_{i+1}\text{,}%
\]
hence $(x_{0},x_{N+1})\in A_{i}\times A_{i+1}\subseteq R$, which finally leads
to $\alpha(x_{0},x_{N+1})\geq1$.

Clearly, $T$ is $\alpha$--contractive of Meir--Keeler type, by (F3).

We claim next that $T$ is $\alpha$--admissible, i.e., (A1) is satisfied.
Indeed, let $x,y\in Y$ such that $\alpha(x,y)\geq1$; hence, there exists
$i\in\{1,2\dots,N\}$ such that $x\in A_{i},y\in A_{i+1}$. Then, by (F1),
$\left(  Tx,Ty\right)  \in\left(  A_{i+1},A_{i+2}\right)  \subseteq R$, hence
$\alpha\left(  Tx,Ty\right)  \geq1$.

Now, let $x_{0}\in A_{1}$ arbitrary. Then $Tx_{0}\in A_{2}$, hence
$\alpha(x_{0},Tx_{0})\geq1$ which concludes (A2).

Next, we prove (A4), by showing that $(Y,d)$ is $\alpha$--regular, so let
$\{x_{n}\}$ be a sequence in $Y$ such that
\[
x_{n}\rightarrow x\in Y\text{ as }n\rightarrow\infty\quad\text{and\quad}%
\alpha(x_{n},x_{n+1})\geq1\text{ for all }n\in\mathbb{N}\text{.}%
\]
It follows that there exist $i,j\in\{1,\ldots,N\}$ such that
\[
x_{n}\in A_{i+n}\text{ for all }n\in\mathbb{N}\text{\quad and\quad}x\in
A_{j},
\]
hence
\[
x_{(j-i-1+N)+kN}\in A_{j-1+(k+1)N}=A_{j-1}\quad\text{for all }k\in\mathbb{N};
\]
By letting
\[
n(k):=(j-i-1+N)+kN\quad\text{for all }k\in\mathbb{N}\text{,}%
\]
note that $j-i-1+N\geq0$, and we conclude that the subsequence $\left\{
x_{n(k)}\right\}  $ satisfies
\[
(x_{n(k)},x)\in A_{j-1}\times A_{j}\subseteq R\quad\text{for all }%
k\in\mathbb{N}%
\]
hence $\alpha(x_{n(k)},x)\geq1$ for all $k$, which proves our claim.

Now, all the conditions in Theorem \ref{Th:2} (for $(Y,d)$, $\alpha$ and $T$)
are satisfied, hence there exists a fixed point $x^{\ast}\in Y$ of $T$.
Clearly, $x^{\ast}\in\bigcap_{i=1}^{N}A_{i}$, since%
\[
x^{\ast}\in A_{k}\text{ for some }k\in\{1,2,\ldots,N\}
\]
and%
\[
x^{\ast}\in A_{i}\Rightarrow x^{\ast}=Tx^{\ast}\in A_{i+1}\text{ for all
}i\text{.}%
\]

Moreover, it is straightforward to check that $Y$ is $\alpha$--connected,
i.e., (A5) is satisfied. Indeed, if $x,y\in Y$ ($x\neq y$) with $x\in A_{i}$,
$y\in A_{j}$ ($i,j\in\{1,2,\ldots,N\}$), then let $z_{0}:=x$, $z_{k}\in
A_{k+i}$ arbitrary for every $k\in\{1,2,\ldots,N+j-i-1\}$ and $z_{N+j-i}:=y$.
Note that $N+j-i\geq1$. Then $(z_{k-1},z_{k})\in R$ (i.e., $\alpha
(z_{k-1},z_{k})\geq1$) for every $k\in\{1,2,\ldots,N+j-i\}$, hence
$(z_{0},z_{1},\ldots,z_{N+j-i})$ is a $\alpha$-chain from $x$ to $y$.

Now, the rest of the conclusion follows by Theorem \ref{Th:3}.
\end{proof}

%%%%%%%%%%%%%%%%%%%%%%%%%%%%%%%%%%%%%%%%%%%%

\section{Some consequences in ordered metric spaces}

Clearly, the initial result of Meir and Keeler (Theorem \ref{Th:0}) follows as
a particular case of our Theorems \ref{Th:2} and \ref{Th:3}, by simply
choosing $\alpha(x,y)=1$ for all $x,y\in X$. In what follows, we will also
show that several fixed point and coupled fixed point results in ordered
metric spaces can be easily deduced (and improved) from our theorems.

\subsection{Fixed point results in ordered metric spaces}

Let $X$ be a nonempty set. Recall that a binary relation $\preceq$ over $X$ is
called a partial order if it is reflexive, transitive and anti-symmetric. If
$\preceq$ is a partial order over $X$, then $x,y\in X$ are called comparable
(subject to $\preceq$) if $x\preceq y$ or $y\preceq x$. Also, $X$ is called
$\preceq$--connected if for every $x,y\in X$, there exist $z_{0},z_{1}%
,\ldots,z_{n}\in X$ such that $z_{0}=x$, $z_{n}=y$ and $z_{i-1},z_{i}$ are
comparable for every $i\in\{1,2,\ldots,n\}$.

In \cite{Harjani2011}, Harjani et al. obtained several fixed point results in
partially ordered sets for mappings satisfying some contraction condition of
Meir--Keeler type. The main results in \cite{Harjani2011} for the case of
nondecreasing mappings can be summarized as follows.

\begin{theorem}
[\cite{Harjani2011}]\label{Th:C1} Let $(X,d)$ be a complete metric space,
$\preceq$ a partial order over $X$ and $T:X\rightarrow X$ such that for all
$\varepsilon>0$ there exists $\delta(\varepsilon)>0$ for which:%
\[
x,y\in X:x\preceq y,~\varepsilon\leq d(x,y)<\varepsilon+\delta(\varepsilon
)\Rightarrow d(Tx,Ty)<\varepsilon\text{.}%
\]
Assume that:

\begin{enumerate}
\item[\textrm{(G1)}] $T$ is nondecreasing (subject to $\preceq$);

\item[\textrm{(G2)}] there exists $x_{0}\in X$ such that $x_{0}\preceq Tx_{0}$.
\end{enumerate}

If either

\begin{enumerate}
\item[\textrm{(G3)}] $T$ is continuous,
\end{enumerate}

\noindent or

\begin{enumerate}
\item[\textrm{(G4)}] for every nondecreasing sequence $\{x_{n}\}$ in $X$ such
that $x_{n}\rightarrow x\in X$, there exists a subsequence $\{x_{n(k)}\}$ of
$\{x_{n}\}$ such that $x_{n(k)}\preceq x$ for all $k\in\mathbb{N}$,
\end{enumerate}

\noindent then $T$ has a fixed point. In addition, if

\begin{enumerate}
\item[\textrm{(G5)}] for every $x,y\in X$, there exists $z\in X$ which is
comparable to $x$ and $y$,
\end{enumerate}

\noindent then the fixed point of $T$ is unique.
\end{theorem}

As it can be easily seen, this result follows straight from Corollary
\ref{Cor:5}, with $\mathcal{R}$ being the partial order $\preceq$. Moreover,
(G5) can be replaced by the weaker assumption:

\begin{enumerate}
\item[\textrm{(G5a)}] $X$ is $\preceq$--connected.
\end{enumerate}

Also, if $x^{\ast}$ is the unique fixed point of $T$, then $T^{n}%
(x)\rightarrow x^{\ast}$ (as $n\rightarrow\infty$) for every $x\in X$. This
follows by Corollary \ref{Cor:5} and its an extension of the conclusion in
Theorem \ref{Th:C1}.

%%%%%%%%%%%%%%%%%%%%%%%

\subsection{Coupled fixed point results in ordered metric spaces}

In \cite{Samet2010}, Samet studied the coupled fixed points of mixed strict
monotone mappings that satisfied a contraction condition of Meir--Keeler type,
thereby extending the previous work of Bhaskar and Lakshmikantham
\cite{Bhaskar2006}. In what follows we present an extension of the results of
Samet \cite{Samet2010}; in this direction, we do not require that the mixed
monotone property be strict and we also weaken other assumptions. We also
improve the conclusion.

First, recall the following definition:

\begin{definition}
[\cite{Bhaskar2006}]Let $(X,\preceq)$ be a partially ordered set. A mapping
$F:X\times X\rightarrow X$ is said to have the mixed monotone property if
\[
x_{1},x_{2},y_{1},y_{2}\in X:x_{1}\preceq x_{2},~y_{1}\succeq y_{2}%
\Longrightarrow F(x_{1},y_{1})\preceq F(x_{2},y_{2}).
\]

\end{definition}

Our extension of the main results in \cite{Samet2010} follows straight from
Corollary \ref{Cor:6}, with $\mathcal{R}$ being the partial order $\preceq$,
and can be stated as follows.

\begin{theorem}
\label{Th:C2}Let $(X,d)$ be a complete metric space, $\preceq$ a partial order
over $X$ and $F:X\times X\rightarrow X$ such that for every $\varepsilon>0$
there exists $\delta(\varepsilon)>0$ for which:
\[
x,y,u,v\in X:x\preceq u,~y\succeq v,~\varepsilon\leq\frac{1}{2}%
[d(x,u)+d(y,v)]<\varepsilon+\delta(\varepsilon)\Rightarrow
d(F(x,y),F(u,v))<\varepsilon.
\]
Suppose that:

\begin{enumerate}
\item[\textrm{(H1)}] $F$ has the mixed monotone property;

\item[\textrm{(H2)}] there exist $x_{0},y_{0}\in X$ such that $x_{0}\preceq
F(x_{0},y_{0})$ and $y_{0}\succeq F(y_{0},x_{0})$.
\end{enumerate}

If either

\begin{enumerate}
\item[\textrm{(H3)}] $F$ is continuous,
\end{enumerate}

\noindent or

\begin{enumerate}
\item[\textrm{(H4)}] $(X,d,\preceq)$ has the following property: if
$\{x_{n}\}$ is a nondecreasing (respectively, nonincreasing) sequence in $X$
such that $x_{n}\rightarrow x$, then $x_{n}\preceq x$ (respectively,
$x_{n}\succeq x$) for all $n$,
\end{enumerate}

\noindent then $F$ has a coupled fixed point $(x^{\ast},y^{\ast})\in X\times
X$. In addition, if

\begin{enumerate}
\item[\textrm{(H5)}] $X$ is $\preceq$--connected,
\end{enumerate}

\noindent then $x^{\ast}=y^{\ast}$, $(x^{\ast},x^{\ast})$ is the unique
coupled fixed point of $F$, $x^{\ast}$ is the unique fixed point of $F$ and
$F^{n}(x,y)\rightarrow x^{\ast}$ as $n\rightarrow\infty$ for all $x,y\in X$.
\end{theorem}

%%%%%%%%%%%%%%%%%%%%%%%%%%%%

\section{Application to a third order two point boundary value problem}

We study the existence and uniqueness of solution to the third order
differential equation
\begin{equation}
x^{\prime\prime\prime}(t)+f(t,x(t))=0,\quad t\in(0,1),\label{eq:20}%
\end{equation}
where $f\in C([0,1]\times\mathbb{R},\mathbb{R})$, with the boundary value
conditions
\begin{equation}
x(0)=x(1)=x^{\prime\prime}(0)=0.\label{eq:20a}%
\end{equation}
This problem is equivalent to finding a solution $x\in C([0,1],\mathbb{R})$ to
the integral equation
\[
x(t)=\int_{0}^{1}G(t,s)f(s,x(s))\,\mathrm{d}s,\
\]
where
\[
G(t,s)=\left\{
\begin{array}
[c]{ll}%
\frac{1}{2}(1-t)(t-s^{2}), & 0\leq s\leq t\leq1,\\[3mm]%
\frac{1}{2}t(1-s)^{2}, & 0\leq t\leq s\leq1.
\end{array}
\right.
\]
Clearly, $G(t,s)\geq0$ for all $t,s\in\lbrack0,1]$. Also, we can verify easily
that
\begin{equation}
\int_{0}^{1}G(t,s)\,\mathrm{d}s=\frac{t-t^{3}}{6}\leq\frac{\sqrt{3}}%
{27}\text{\quad for all }t\in\lbrack0,1].\label{eq:21}%
\end{equation}

Let $\Phi$ be the set of all nondecreasing functions $\varphi:[0,+\infty
)\rightarrow\lbrack0,+\infty)$ such that for all $\varepsilon>0$ there exists
$\delta(\varepsilon)>0$ with
\[
\varepsilon\leq t<\varepsilon+\delta(\varepsilon)\Longrightarrow
\varphi(t)<\varepsilon.
\]

Let $\xi:\mathbb{R}^{2}\rightarrow\mathbb{R}$ and $\varphi\in\Phi$. We
consider the following assumptions:

\begin{enumerate}
\item[\textrm{(J1)}] there exists $N\in\mathbb{N}\setminus\{0\}$ such that
\[
a_{0},a_{1},\dots,a_{N+1}\in\lbrack0,1]:\xi(a_{i},a_{i+1})\geq0\text{ for all
}i\in\{0,1,\ldots,N\}\Longrightarrow\xi(a_{0},a_{N+1})\geq0.
\]

\item[\textrm{(J2)}] for every $a,b\in\mathbb{R}$:
\[
\xi(a,b)\geq0\Longrightarrow\left\vert f(t,a)-f(t,b)\right\vert \leq9\sqrt
{3}\varphi(\left\vert a-b\right\vert )\quad\text{for all }t\in\lbrack0,1].
\]

\item[\textrm{(J3)}] for every $x,y\in C\left(  [0,1]\right)  $:
\[
\inf_{t\in\lbrack0,1]}\xi(x(t),y(t))\geq0\Longrightarrow\inf_{t\in\lbrack
0,1]}\xi\left(  \int_{0}^{1}G(t,s)f(s,x(s))\,\mathrm{d}s,\int_{0}%
^{1}G(t,s)f(s,y(s))\,\mathrm{d}s\right)  \geq0.
\]

\item[\textrm{(J4)}] there exists $x_{0}\in C\left(  [0,1]\right)  $ such
that
\[
\inf_{t\in\lbrack0,1]}\xi\left(  x_{0}(t),\int_{0}^{1}G(t,s)f(s,x_{0}%
(s))\,\mathrm{d}s\right)  \geq0
\]

\item[\textrm{(J5)}] for every $x,y\in C\left(  [0,1]\right)  $, there exist
$z_{0},z_{1},\ldots,z_{n}\in C\left(  [0,1]\right)  $ such that $z_{0}=x$,
$z_{n}=y$ and, for every $i\in\{1,2,\ldots,n\}$:
\[
\inf_{t\in\lbrack0,1]}\xi(z_{i-1}(t),z_{i}(t))\geq0\quad\text{or\quad}%
\inf_{t\in\lbrack0,1]}\xi(z_{i}(t),z_{i-1}(t))\geq0.
\]

\end{enumerate}

\begin{theorem}
\label{Th:A1}Let $f:[0,1]\times\mathbb{R}\rightarrow\mathbb{R}$ be continuous
and assume that there exist $\xi:\mathbb{R}^{2}\rightarrow\mathbb{R}$ and
$\varphi\in\Phi$ such that (J1)--(J4) are satisfied. Then the equation
(\ref{eq:20}) with the boundary conditions (\ref{eq:20a}) has solution. In
addition, if (J5) is satisfied, then the solution is unique.
\end{theorem}

\begin{proof}
Let $X:=C\left(  [0,1]\right)  $ be endowed with the metric
\[
d(u,v)=\max_{t\in\lbrack0,1]}|u(t)-v(t)|,\quad u,v\in X.
\]
It is well known that $(X,d)$ is a complete metric space. Define the mapping
$T:X\rightarrow X$ by
\[
(Tx)(t)=\int_{0}^{1}G(t,s)f(s,x(s))\,\mathrm{d}s\quad(x\in X,t\in\lbrack0,1]).
\]
The problem reduces to the fixed point problem for $T$.

Let $\alpha:X\times X\rightarrow\lbrack0,\infty)$ be defined by%
\[
\alpha(x,y)=\left\{
\begin{array}
[c]{ll}%
1, & \text{if }\xi(x(t),y(t))\geq0\text{ for all }t\in\lbrack0,1],\\
0, & \text{otherwise.}%
\end{array}
\right.
\]
It is easy to observe that $\alpha$ is $N$--transitive by (J1), $T$ is
$\alpha$--admissible by (J3) and $\alpha(x_{0},Tx_{0})\geq1$ by (J4). Also, it
follows in a standard fashion that $T$ is continuous, hence we omit this proof.

Now, using (J2), (\ref{eq:21}) and the fact that $\varphi$ is nondecreasing,
it follows that for all $x,y\in X$ with $\alpha(x,y)\geq1$:
\[
\left\vert (Tx)(t)-(Ty)(t)\right\vert \leq\int_{0}^{1}G(t,s)\left\vert
f(s,x(s))-f(s,y(s))\right\vert \,\mathrm{d}s\leq9\sqrt{3}\left(  \int_{0}%
^{1}G(t,s)\,\mathrm{d}s\right)  \varphi(d(x,y))\leq\varphi(d(x,y))\text{,}%
\]
hence%
\[
d\left(  Tx,Ty\right)  \leq\varphi(d(x,y))\text{\quad for all }x,y\in X\text{
with }\alpha(x,y)\geq1\text{.}%
\]
This clearly leads to
\begin{equation}
\alpha(x,y)d\left(  Tx,Ty\right)  \leq\varphi(d(x,y))\text{\quad for all
}x,y\in X\text{.}\label{eq:22}%
\end{equation}
Now, let $\varepsilon>0$. Since $\varphi\in\Phi$, there exists $\delta
(\varepsilon)>0$ such that
\begin{equation}
\varepsilon\leq a<\varepsilon+\delta(\varepsilon)\Longrightarrow
\varphi(a)<\varepsilon.\label{eq:23}%
\end{equation}
Let $x,y\in X$ with $\varepsilon\leq d(x,y)<\varepsilon+\delta(\varepsilon)$.
Then, by (\ref{eq:22}) and (\ref{eq:23}), it follows that%
\[
\alpha(x,y)d\left(  Tx,Ty\right)  \leq\varphi(d(x,y))<\varepsilon;
\]
hence, we conclude that $T$ is $\alpha$--contractive mapping of Meir-Keeler
type.\newline

Now, we can apply Theorem \ref{Th:1} and obtain the existence of a fixed point
of $T$, hence the existence of a solution to (\ref{eq:20})--(\ref{eq:20a}). In
addition, (J5) ensures that $X$ is $\alpha$--connected and the uniqueness of
the solution follows by Theorem \ref{Th:3}. The proof is now complete.
\end{proof}

\begin{corollary}
Let $f:[0,1]\times\mathbb{R}\rightarrow\mathbb{R}$ be continuous and assume
there exists $\varphi\in\Phi$ such that the following conditions are satisfied:

\begin{enumerate}
\item[\textrm{(K1)}] $0\leq f(t,b)-f(t,a)\leq9\sqrt{3}\varphi(b-a)$ for all
$t\in\lbrack0,1]$ and $a,b\in\mathbb{R}$ with $a\leq b$.

\item[\textrm{(K2)}] there exists $x_{0}\in C\left(  [0,1]\right)  $ such that
for all $t\in\lbrack0,1]$, we have
\[
x_{0}(t)\leq\int_{0}^{1}G(t,s)f(s,x_{0}(s))\,\mathrm{d}s.
\]

\end{enumerate}

Then (\ref{eq:20})-(\ref{eq:20a}) has a unique solution.
\end{corollary}

\begin{proof}
Consider the mapping $\xi:\mathbb{R}^{2}\rightarrow\mathbb{R}$ be defined by
$\xi(a,b)=b-a$ ($a,b\in\mathbb{R}$). Then the result follows straight from
Theorem \ref{Th:A1}. Indeed, $\xi$ clearly satisfies (J1), while (J2) and (J3)
follow by (K1). Condition (K2) ensures (J4), while (J5) follows easily, by
noting that for every $x,y\in C\left(  [0,1]\right)  $, the function%
\[
z:[0,1]\rightarrow\mathbb{R}\text{,\quad}z(t)=\max\{x(t),y(t)\}\text{ }%
(t\in\lbrack0,1])
\]
satisfies%
\[
z\in C\left(  [0,1]\right)  ,~\inf_{t\in\lbrack0,1]}\xi(x(t),z(t))\geq
0,~\inf_{t\in\lbrack0,1]}\xi(y(t),z(t))\geq0\text{.}%
\]

\end{proof}

\begin{remark}
Condition (K2) can be replaced by

\begin{enumerate}
\item[\textrm{(K2a)}] there exists $x_{0}\in C\left(  [0,1]\right)  $ such
that for all $t\in\lbrack0,1]$, we have
\[
x_{0}(t)\geq\int_{0}^{1}G(t,s)f(s,x_{0}(s))\,\mathrm{d}s,
\]
while all the other conditions and conclusions remain unchanged. In this case,
the proof follows similarly, by letting $\xi:\mathbb{R}^{2}\rightarrow
\mathbb{R}$ be defined by $\xi(a,b)=a-b$ ($a,b\in{\mathbb{R}}$).
\end{enumerate}
\end{remark}

%%%%%%%%%%%%%%%%%%%%%%%%%%%%%%%%
\section*{Acknowledgement}

The second author is grateful for the financial support provided by the Sectoral
Operational Programme Human Resources Development 2007-2013 of the Romanian
Ministry of Labor, Family and Social Protection through the Financial
Agreement POSDRU/89/1.5/S/62557.
%%%%%%%%%%%%%%%%%%%%%%%%%%%%%%%%

\end{document}